\documentclass[journal]{IEEEtran}

\ifCLASSINFOpdf
\else
   \usepackage[dvips]{graphicx}
\fi
\usepackage{url}

\usepackage{graphicx}
\usepackage{amssymb,amsmath,amsthm}
\usepackage{bm,mathrsfs}
\usepackage{enumerate}
\usepackage{hyperref}
\usepackage{cleveref}
\usepackage{graphicx,tikz}
\usepackage{algorithm,algpseudocode}
\usepackage[tight,footnotesize]{subfigure}
\usepackage{multirow}
\usepackage{rotating}
\usepackage{dsfont}
\usepackage{booktabs}
\usepackage{color}

\DeclareMathOperator*{\argmin}{arg\,min}
\providecommand{\abs}[1]{\left\lvert#1\right\rvert}
\providecommand{\norm}[1]{\left\rVert#1\right\rVert}
\providecommand{\set}[1]{ \left\{ #1  \right\}  }
\providecommand{\setb}[2]{ \left\{ #1 \ \middle| \  #2 \right\}  }

\providecommand{\parenth}[1]{\left( #1 \right) }

\newtheorem{theorem}{Theorem}

\newtheorem{proposition}{Proposition}

\newtheorem{remark}{Remark}
\newtheorem{definition}{Definition}

\begin{document}

\title{Fast and accurate log-determinant \\ approximations }

\author{Owen Deen, 
Colton River Waller, 
and
John Paul Ward

\thanks{
This work was supported in part by the NSF under Award No. 1719498 and by the NSA under Award No. H98230-23-1-008
}
\thanks{
J.P. Ward is with the Department of Mathematics and Statistics, North Carolina A\&T State University, Greensboro, NC 27408 USA (e-mail: jpward@ncat.edu). 
}

\thanks{
The authors contributed equally to this work and are listed alphabetically.
}

\thanks{Python code is available on 
\href{https://arxiv.org/}{arXiv} and
\href{https://github.com/owendeen/NCAT_REU_2023}{Github}. 
}
}
\maketitle

\begin{abstract}
We consider the problem of estimating log-determinants of large, sparse, positive definite matrices. A key focus of our algorithm is to reduce computational cost, and it is based on sparse approximate inverses. The algorithm can be implemented to be adaptive, and it uses graph spline approximation to improve accuracy. We illustrate our approach on classes of large sparse matrices.
\end{abstract}

\begin{IEEEkeywords}
graph models, graph splines, log-determinant, positive definite, sparse matrices
\end{IEEEkeywords}

\ifCLASSOPTIONpeerreview
 \begin{center} \bfseries EDICS Category: DSP-SAMP, DSP-SPARSE  \end{center}
\fi

\IEEEpeerreviewmaketitle

\section{Introduction}
\label{sec:introduction}

\IEEEPARstart{I}{n} recent years, log-determinants have been used in a variety of applications. For example, log-determinant divergences are used on spaces of positive definite matrices for image processing, pattern recognition, and computer vision \cite{mei14,cichocki15,peng2022hyperspectral}. Other related applications include rank approximation for Robust PCA and subspace clustering \cite{peng2015subspace,peng2020robust,liu21}.

We propose an algorithm for approximating log-determinants using sparse approximate inverses. 
It builds upon the work \cite{reusken02}. 
Rather than constructing a single large approximate inverse, we compute a sequence of smaller approximate inverses and use the trend of the data to define the final approximation. The smaller approximations appear to develop a clear pattern that we can exploit, cf. \Cref{fig:compare_algs}. 
In this work, we propose using graph-based splines to fit the data and increase the accuracy of the prediction \cite{pesenson10s,shuman13,ward20}.

In \cite[Section 5]{reusken02}, the author discusses the computational complexity  for sparse approximate inverses, and he shows that the cost of a single approximation is less than that for Monte Carlo approaches.  Our technique of computing smaller inverses and fitting the data with a graph further reduces the cost and exhibits higher accuracy. 

The two key advantages to using sparse approximate inverses over competing approaches are reduced memory requirements and parallel computations. To compute a row of a sparse approximate inverse, we only need access to the rows and columns of the original matrix that correspond to the entries being created. Also, we do not need to store the sparse approximate inverse for the approximation of the log-determinant.

\subsection{Background: positive definite matrices}

This section covers relevant properties of (sparse) positive definite matrices. For additional details, we refer the reader to \cite{horn12,watkins04}.

Throughout this work, we consider positive definite matrices $A$ of size $n\times n$. 
The determinant of a matrix $A$ is denoted $\abs{A}$. 
We define an \textit{index set} of an $n\times n$ matrix $A$ to be a subset of $\set{1,2,\dots, n}$. We are primarily interested in nonempty index sets where the indices are increasing. 
The submatrix of $A$ that is formed by restricting to the rows and columns from an index set $\alpha$ is denoted $A(\alpha)$. 
We define $\abs{A(\emptyset)} =1$.

\begin{proposition}\label{prop:unique_LU}
For a positive definite matrix $A$, there are unique matrices $L$, which is lower triangular, and $U$, which is unit upper triangular, such that $A=LU$.
\end{proposition}

\begin{proposition}
For a positive definite matrix $A$, 
there is a unique upper triangular matrix $R$, with positive diagonal, 
such that
$A=R^\top R$. 
$R$ is called the Cholesky factor of $A$.
\end{proposition}

Note that these decompositions are closely connected. 
If $D$ is the diagonal of $L$, 
then $A=(LD^{-1/2})(D^{1/2}U)$ gives the Cholesky factorization of $A$.

\begin{proposition}[Hadamard-Fischer Inequalities \cite{johnson85}]\label{prop:hadamard}
Let $\alpha,\beta$ be index sets for a positive definite matrix $A$.  
Then
\begin{equation}
\abs{A \parenth{\alpha \bigcup \beta}}
\leq
\frac{\abs{A(\alpha)} \abs{A(\beta)}}{\abs{A(\alpha \bigcap \beta)}}
\end{equation}
or equivalently
\begin{equation}
\frac{\abs{A \parenth{\alpha \bigcap \beta}}}{\abs{A(\alpha)}}
\leq
\frac{ \abs{A(\beta)}}{\abs{A(\alpha \bigcup \beta)}}
\end{equation}
\end{proposition}

\subsection{Background: graphs and splines}

We consider finite graphs of the form 
$\mathcal{G} = \set{\mathcal{V},E\subset \mathcal{V}\times \mathcal{V},w}$.
The vertex set $\mathcal{V}$ contains at least two nodes. The edge set is $E$. We assume that there are no multiple edges nor loops in the graph.  The weight function $w:E\rightarrow \mathbb{R}_{>0}$ specifies the closeness of two vertices.

The adjacency matrix $A$ is an $n\times n$ matrix, where $n$ is the number of vertices in the graph. The columns and rows are indexed by the vertices. The $i,j$ entry is zero for pairs of vertices that are not connected by an edge. If vertex $i$ and vertex $j$ are connected by an edge, the $i,j$ entry is the weight of the edge connecting them.

The degree matrix $D$ is a diagonal matrix whose entries are the row sums of the adjacency matrix. The Laplacian on the graph is the matrix $\mathcal{L}=D-A$. 

If a function $g$, living on a graph, is known at some vertices $\mathcal{V}_k$, and unknown at others $\mathcal{V}_u$, then we can approximate the unknown values using a 
polyharmonic
graph spline. The approximation is computed to minimize the norm of the Laplacian applied to the function. Let $\mathcal{L}_k,\mathcal{L}_u$ represent the Laplacian restricted to the columns corresponding to the vertices of $\mathcal{V}_k,\mathcal{V}_u$, respectively. If  $g_k$ are the known values of $g$, then the spline approximation $g_u$ to $g$ on $\mathcal{V}_u$ is defined by 
\begin{equation}
g_u 
= 
\argmin_{f} 
\norm{\mathcal{L}_u f + \mathcal{L}_k g_k}_2^2
\end{equation}
which can be solved explicitly using the normal equations
\begin{equation}
\mathcal{L}_u^\top \mathcal{L}_u g_u 
= 
\mathcal{L}_u^\top \mathcal{L}_k g_k
\end{equation}
Further details are given in \Cref{subsec:sparsity}, and an algorithm is provided in \Cref{subsec:algorithm_spline}.

\section{Sparse approximate inverses}
\label{sec:sparsity}

For the benefit of the reader, in the first part of this section we recall the construction of sparse approximate inverses as described in \cite{reusken02}. 

Consider the $LU$-decomposition $A=LU$, which can be written as $L^{-1}A=U$. 
We approximate $L^{-1}$ by 
$\hat{L}^{-1}$, which is lower triangular and has a specified 
sparsity pattern. 
The corresponding approximation to $U$ is $\hat{U}$, which is equivalent to the identity matrix on the given sparsity pattern. These conditions determine a unique approximation \cite[Lemma 3.1]{reusken02}. 

To make this more precise, a sparsity pattern $E$ for a lower triangular matrix is a subset of 
$\setb{(i,j)}{1\leq j\leq i \leq n}$,
where we require $(i,i)$ to be in $E$ for all $i$.
These are the locations where $\hat{L}^{-1}$ is allowed to be nonzero. We require $\hat{L}^{-1}$ to be 0 on the complement of $E$.
Then $\hat{L}^{-1}$ is constructed from the defining equations
\begin{equation}\label{eq:def_Lhat}
\parenth{\hat{L}^{-1}A}_{i,j} =I_{i,j},
\quad
\forall (i,j)\in E
\end{equation}
where $I$ is the $n\times n$ identity. 
Note that  
$\hat{L}$ is lower triangular, while $\hat{U}$ may not be upper triangular.

We can compute $\hat{L}^{-1}$ one row at a time from \eqref{eq:def_Lhat}. 
\begin{definition}
We use the following notation:
\begin{itemize}
\item
$E_i$ : subset of $E$ corresponding to row $i$
\item 
$b_i$ : $i$th standard basis (column) vector in $\mathbb{R}^n$ restricted to $E_i$
\item 
$A_i$ : submatrix of $A$ with rows and columns from $E_i$ 
\item 
$\ell_i$ : row $i$ of $\hat{L}^{-1}$ restricted to $E_i$ (viewed as a column vector)
\end{itemize}
\end{definition}

Note that $A_i$ can be viewed as $P A P^\top$, where $P$ is a projection matrix corresponding to $E_i$.

Considering \eqref{eq:def_Lhat}, $\ell_i$ is defined by 
\begin{equation}
\ell_i^\top A_i = b_i^\top
\end{equation}
or equivalently
\begin{equation}\label{eq:row_computation}
A_i \ell_i = b_i
\end{equation}

\begin{remark}
In the construction of $\ell_i$, we only need the submatrix $A_i$, and the system to be solved can be substantially smaller than the size of $A$. As noted in \cite{reusken02}, these properties can be exploited for reduced memory requirements as well as parallelization. 
\end{remark}

\subsection{Sparsity model}
\label{subsec:sparsity}

In \cite{reusken02}, the author used a single sparse approximate inverse for the approximation of $\abs{A}^{1/n}$. Our method, detailed below, is an extension using more than one sparsity pattern in the approximation of the log-determinant of a matrix $A$. 

We consider all possible sparsity patterns as living on a graph $\mathcal{G}$. The vertices are the sparsity patterns. Edges exist between sparsity patterns 
$E\subset \bar{E}$ if $E$ contains exactly one less entry than $\bar{E}$.

We define the 
density
of a pattern  $E$ to be the number of entries divided by the total number possible, $n(n+1)/2$. The distance between adjacent patterns is the absolute value of the difference in their 
densities.

Rather than work with the full graph, we typically only compute values at a few vertices, so we collapse the graph to the known vertices and a vertex to be approximated. The distances are inherited from the full graph. The weights of the edges are the inverses of the distances on the smaller graph.

\subsection{Properties of sparse approximate inverses}

To approximate the log-determinant, we use the negative of the log-determinant of $\hat{L}^{-1}$. This means we only need the diagonal of $\hat{L}^{-1}$ in our approximation.

Considering \eqref{eq:row_computation}, we present three approaches to finding the diagonal element of $\hat{L}^{-1}$ in row $i$. 

First, we can formulate this as 
\begin{align}
\begin{split}
\ell_i &= A_i^{-1}b_i\\
b_i^\top \ell_i &= b_i^\top A_i^{-1}b_i\\
\end{split}
\end{align}
so the diagonal element of $\hat{L}^{-1}$ (the last entry in $\ell_i$) is the bottom right entry of $A_i^{-1}$.

Another possibility is to use the LU-decomposition 
$A_i=L_iU_i$. In this case, we have
\begin{align}
\begin{split}
A_i^{-1} &= U_i^{-1}L_i^{-1}\\
b_i^\top A_i^{-1} b_i
&= 
b_i^\top  U_i^{-1}L_i^{-\top} b_i\\ 
&= 
b_i^\top L_i^{-1} b_i\\
\end{split}
\end{align}
so the last entry in $\ell_i$ is the bottom right entry of $L_i^{-1}$, which is also the inverse of the bottom right entry of $L_i$.

Finally, using the Cholesky factorization 
$A_i=R_i^\top R_i$, we have
\begin{align}
\begin{split}
A_i^{-1} &= R_i^{-1}R_i^{-\top}\\
b_i^\top A_i^{-1} b_i
&= 
b_i^\top  R_i^{-1}R_i^{-\top} b_i\\
b_i^\top A_i^{-1} b_i 
&= 
\parenth{R_i^{-\top} b_i}^\top 
\parenth{R_i^{-\top} b_i}\\
\end{split}
\end{align}
Using the triangular structure of $R_i$, the diagonal element of $\hat{L}^{-1}$ is the square of the inverse of the bottom right entry of the Cholesky factor of $A_i$. 

Note that the last is the formulation proposed in \cite{reusken02}. 
We use the penultimate version in our algorithm. 
The first is of theoretical value, and it is used in our monotonicity proof below.

\begin{proposition}\label{prop:exact}
If $E=\setb{(i,j)}{1\leq j\leq i \leq n}$, then the resulting $\hat{L}^{-1}$ is equal to $L^{-1}$. Hence the corresponding approximation is exact.
\end{proposition}
\begin{proof}
In this case, we have $\hat{L}^{-1}A=\hat{U}$ where $\hat{L}$ is lower triangular and $\hat{U}$ is unit upper triangular. The result now follows from the uniqueness of the $LU$ decomposition of $A$,
\Cref{prop:unique_LU}.
\end{proof}

Every sparsity pattern results in an over approximation to the log-determinant, so smaller approximations are more accurate \cite[Remark 3.11]{reusken02}. 
\begin{theorem}\label{theorem:monotonicity}
Let $E\subset \bar{E}$ be sparsity patterns. Then the approximation to the log-determinant of $A$ using $\bar{E}$ is smaller than or equal to the approximation corresponding to $E$.
\end{theorem}
\begin{proof}
We consider the computation of the diagonal of $\hat{L}^{-1}$ on row $i$. We have 
sets of column indices $\mu,\nu$ for $E,\bar{E}$ on row $i$, respectively. 

We address the general situation where $\mu$ contains at least two indices. The result is based on \Cref{prop:hadamard} together with the adjoint construction of the inverse of a matrix. 
The bottom-right entry of $A_{i}^{-1}$, corresponding to $\mu$, is
\begin{equation}
\frac{\abs{A(\mu \backslash \set{i})}}{\abs{A(\mu)}}
\end{equation}
Similarly, the bottom-right entry of $A_{i}^{-1}$, corresponding to $\nu$, is
\begin{equation}
\frac{\abs{A(\nu \backslash \set{i})}}{\abs{A(\nu)}}
\end{equation}
The inequality now follows from \Cref{prop:hadamard} by setting \\
$\alpha = \mu$ and $\beta = \nu \backslash \set{i}$.
\end{proof}

\begin{remark}
\Cref{theorem:monotonicity} can alternatively be shown using the results of \cite[Section 3.3]{reusken02} 
\end{remark}

\section{Algorithm}
\label{sec:algorithm}

\begin{definition}
The primary components of the algorithm are the following: 
\begin{itemize}
\item
$E^j$ : sparsity pattern, stored as a boolean array, corresponding to the nonzero entries of $A^j$ 
(the $j$th power of $A$);
\item 
$D^j$ : approximation to the log-determinant corresponding to the sparsity pattern $E^j$;
\item 
$S^j$ : spline approximation to the log-determinant using $D^1,\dots,D^j$. 
\end{itemize}
\end{definition}

Note that \Cref{theorem:monotonicity} and \Cref{prop:exact} together imply that the approximations $D^j$ are monotonically decreasing to the exact value of the log-determinant. This is the basis for our algorithm.

\subsection{Sparse inverses}
\label{subsec:algorithm_sparse}

The algorithm is split into two steps. We first compute, in parallel, $m$ approximations to the log-determinant, i.e. $D^1,\dots,D^m$. These are created from nested sparsity patterns corresponding to the locations of nonzero elements of $A^j$. This choice of patterns was used in \cite{reusken02}. It is further discussed and justified in \cite{chow00}. 

The advantage to this approach is a savings in memory, as the full sparsity patterns do not need to be stored. Instead, each row is computed as needed from $E^1$.

\begin{algorithm}[ht]
\caption{Sparse approximate inverses}\label{alg:cap}
\begin{algorithmic}
\State \textbf{input} $A \in \mathbb{R}^{n \times n}$
\State \textbf{input} $m\in \mathbb{Z}^+$ \Comment{max power}
\State $E \gets density(A)$ \Comment{boolean array of nonzeros}
\State $D^{1},D^{2},...,D^{m} \gets 0$
\Comment{initialize output}
\For{ $i = 1,\dots,n$}  \Comment{parallel loop}
    \State $\alpha_{i} \gets E(i,:) $ \Comment{row grab, row $i$ of $E^1$}
    \For{$j = 1,\dots,m$}
        \State $\beta_{i}  \gets \alpha_{i}$ 
        \State $\beta_{i}(i+1:n) \gets 0$
        \State $A_{i} \gets A(\beta_i,\beta_i)$
        \Comment{extract submatrix}
        \State $L_{i},U_{i} \gets LU(A_{i})$ \Comment{LU factorization} 
        \State $D^j \gets D^j + \log((L_{i})_{n_i , n_i})$
        \Comment{$L_i$ is $n_i\times n_i$ }
        \State $\alpha_{i} \gets \alpha_i E $ 
        \Comment{update to row $i$ of $E^{j+1}$}
    \EndFor    
\EndFor
\end{algorithmic}
\end{algorithm}

\subsection{Spline Construction}
\label{subsec:algorithm_spline}

After computing $D^1,\dots,D^m$, the second step is to place the computed values on a graph and fit the data (interpolate) with a spline to provide a more accurate approximation to the log-determinant.

We think of the graph as being embedded on the real line, where the vertices live at the 
density
of the corresponding sparsity pattern $E^j$.  The unknown point that is to be approximated is set beyond the highest computed density. In our experiments, we set it to be 1.5 times further than the previous gap. Weights are computed as the inverse of the distance between vertices.

\begin{algorithm}
\caption{Spline interpolation}\label{alg:cap}
\begin{algorithmic}
\State \textbf{input} $D^{1},D^{2},...,D^{m}$
\Comment{approximations} 
\State \textbf{input} $x_{1},x_{2},...,x_{m}$
\Comment{densities of $E^j$} 
\State $x_{m+1} \gets x_m + 1.5(x_m-x_{m-1})$
\Comment{interpolation point}
\State $A,D \gets \mathbf{0}_{M+1\times M+1}$ \Comment{initialize Adjacency, Degree}
\For{$j=1,\dots,m$}
\State $A_{j,j+1} \gets x_{j+1}-x_{j}$
\State $A_{j+1,j} \gets x_{j}-x_{j-1}$
\EndFor
\State $diagonal(D) \gets RowSums(A)$
\Comment{fill diagonal of $D$}
\State $\mathcal{L} \gets D-A$
\Comment{Laplacian matrix}
\State $g_k \gets (D^{1},D^{2},...,D^{m})$
\Comment{known values}
\State $\mathcal{L}_k \gets \mathcal{L}(:\ , 1:m)$ 
\Comment{submatrix for known values}
\State $\mathcal{L}_u \gets \mathcal{L}(:\ , m+1)$ 
\Comment{submatrix for unknown value}
\State $g_u \gets 
\mathcal{L}_u^\top \mathcal{L}_k g_k / 
\mathcal{L}_u^\top \mathcal{L}_u $
\Comment{interpolated value}
\end{algorithmic}
\end{algorithm}

\subsection{Python code notes}

Our experiments were performed using Python code, and we make use of the SciPy package (version 1.10.1) \cite{2020SciPy-NMeth} for linear algebra operations. Parallelization is implemented using \emph{concurrent.futures} (version 3.2).  SciPy's sparse-LU algorithm calls the SuperLU library\cite{li99}, which is written in C. 

\section{Experimental results}
\label{sec:experiment}

Our experiments are mainly focused on the class of Laplacian matrices on $d$-dimensional grids, where the weight between adjacent points is 1. From an infinite matrix, we restrict to the rows and columns corresponding to a cube with side length $N$.  The $N\times N$ matrix in $d$ dimensions 
is denoted 
$\mathcal{L}(N,d)$, e.g. $\mathcal{L}(2,2)$ represents the matrix

\begin{equation}
\begin{bmatrix}
4 & -1 & -1 & 0\\
-1 & 4 & 0 & -1\\
-1 & 0 & 4 & -1\\
0 & -1 & -1 & 4
\end{bmatrix}
\end{equation}

\subsection{Motivating examples}

Here, we show some comparisons with an exact method for computation of the log-determinant of sparse positive definite matrices. The exact algorithm is Sparse-LU from the Python package SciPy \cite{2020SciPy-NMeth}.

First, we consider two matrices and compare computation time, \Cref{tab:compare_compute_time}. The key observations are
\begin{itemize}
\item 
The total time to compute seven decreasing approximations to the log-determinant is less than the time for the Sparse-LU algorithm. 
\item 
There is a clear trend in the approximations 
that we can use to extrapolate and make a more precise approximation.
\item 
Increasing the size of a sparsity pattern increases the cost (more shared memory between processors), so we aim to compute as few approximations as possible and extrapolate the results. 
\end{itemize}

Next, we compare based on memory in \Cref{tab:compare_memory}. 
Again we see substantially lower system requirements.

\begin{table}
\centering
\caption{
Computation time (seconds) for computing $D^1,\dots,D^m$ simultaneously and approximation  $D^m$ are reported. 
}
\label{tab:compare_compute_time}
\small
\setlength{\tabcolsep}{3pt}
\begin{tabular}{c||c|c|c|c}
Matrix 
& \multicolumn{2}{c|}{$\mathcal{L}(15,4)$} 
& \multicolumn{2}{c}{$\mathcal{L}(16,4)$} 
\\ \hline 
Size 
& \multicolumn{2}{c|}{$50,625\times 50,625$}
& \multicolumn{2}{c}{$65,536\times 65,536$} 
\\ \hline 
Density
& \multicolumn{2}{c|}{$1.7e-4$} 
& \multicolumn{2}{c}{$1.3e-4$}   
\\ \hline \hline
$m$ & time & $D^m$ & time & log-det\\ \hline 
1    &14.5&102227.3 &18.8&132319.1\\ \hline 
2    &15.5&101778.7 &20.4&131732.7\\ \hline 
3    &15.9&101665.4 &20.4&131583.8\\ \hline 
4    &16.7&101627.3 &22.1&131533.3\\ \hline 
5    &16.5&101612.3 &21.6&131513.4\\ \hline  
6    &25.5&101605.9 &34.4&131504.7\\ \hline 
7    &57.3&101602.8 &78.6&131500.6\\ \hline 
\hline
Sparse-LU  &328.3& 101599.6 &488.0&131496.0\\
\end{tabular}
\label{tab:compare_compute_time1}
\end{table}

\begin{table}
\centering
\caption{Memory comparison between Scipy's Sparse-LU and our algorithm. In each case, we compute $D^1,\dots,D^4$. Maximum memory used is reported in kilobytes. }
\label{tab:compare_memory}
\small
\setlength{\tabcolsep}{3pt}
\begin{tabular}{c||c|c}
Matrix 
& Our Algorithm 
& Sparse-LU \\ \hline \hline
$\mathcal{L}(15,3)$ & 85,908 & 98,496\\ \hline
$\mathcal{L}(25,3)$ & 136,224 & 348,888\\ \hline
$\mathcal{L}(35,3)$ & 248,744  & 1,334,308\\ \hline
$\mathcal{L}(45,3)$ & 438,696 & 4,617,176\\ \hline \hline
$\mathcal{L}(15,4)$ & 408,904 & 7,510,124\\ 
\end{tabular}
\label{tab:compare_memory1}
\end{table}

\subsection{Spline improvement experiments}
  
In \Cref{tab:compare_algs1}, we see the improvement from using spline approximation for several matrices. 
In \Cref{fig:compare_algs}, we plot several approximations for a single matrix, and we see how the spline approximations $S^j$ are converging to the true value faster and are consistently more accurate than $D^j$ alone.
This figure also shows the trend of the $D^j$ that we exploit.

\begin{table}
\centering
\caption{Improvement using splines. 
Error is reported as relative error. 
Note that $S^3$ is computed  using $D^1,D^2,D^3$.
}
\label{tab:compare_algs}
\small
\setlength{\tabcolsep}{3pt}
\begin{tabular}{c||c|c}
Matrix &  $D^4$ error &  $S^3$ error \\ \hline \hline
$\mathcal{L}(15,3)$&0.11\%&0.002\%\\ \hline
$\mathcal{L}(25,3)$&0.145\%&0.032\%\\ \hline
$\mathcal{L}(35,3)$&0.163\%&0.047\%\\ \hline
$\mathcal{L}(45,3)$&0.173\%&0.057\%\\  \hline \hline
$\mathcal{L}(15,4)$&0.027\%&0.019\%
\end{tabular}
\label{tab:compare_algs1}
\end{table}

\begin{figure}
    \centering
    \centerline{\includegraphics[width=0.45\textwidth]{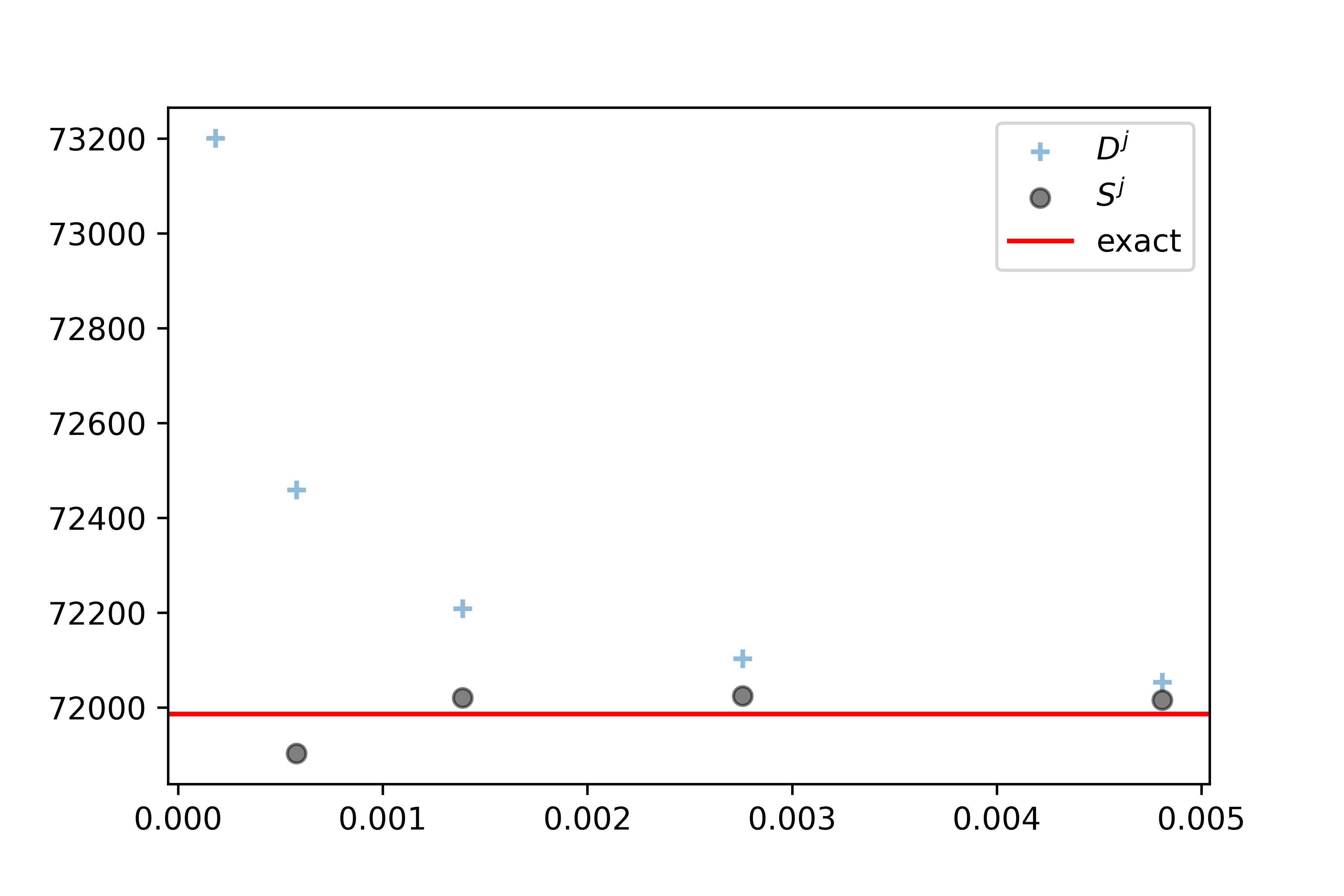}}
    \caption{
    Convergence of the approximations $D^j$ and spline approximations $S^j$ to the exact log-determinant of $\mathcal{L}(35,3)$.
    The $x$-axis represents the 
    density
    of the approximation being considered. 
    }
    \label{fig:compare_algs}
\end{figure}

\section{Summary and discussion}
\label{sec:summary}

We have presented a computationally efficient algorithm for the approximation of the log-determinant of large, sparse, positive definite matrices.  In our experiments, the spline-based approximation uses the smooth decay of the data to more accurately predict the results, reducing computational demands. We have also provided a new proof of the monotonicity of the approximations with respect to sparsity pattern inclusion.

The choice to use graph based splines was based on the discrete structure of the sparsity patterns.  In our current formulation, we use a particular path subgraph, and our data is essentially one-dimensional. However, it is possible that other subgraphs may provide more information, leading to higher accuracy.  Additional considerations include  adaptive sparsity pattern constructions.

The current state of our algorithm shows promise, and we plan further development for future testing. A particular comparison of interest is with Incomplete LU-decomposition
\cite{meijerink77}, which has some similarities to sparse approximate inverses. Essentially, incomplete-LU creates a sparse approximation to $L$ in $A=LU$, rather than a sparse approximation to $L^{-1}$ in $L^{-1}A=U$.
One potential advantage to sparse approximate inverses is that we do not need to store $\hat{L}^{-1}$, and preliminary experiments indicate that we have higher accuracy for comparable memory usage.

\bibliographystyle{IEEEtran}
\bibliography{graph.bib}

\end{document}